\documentclass[11pt,a4paper]{article}

\usepackage{latexsym}
\usepackage[T2A]{fontenc}
\usepackage[utf8x]{inputenc}
\usepackage[russian,english]{babel}

\usepackage[a4paper,top=3cm,bottom=4.5cm,left=2.9cm,right=2.9cm]{geometry}
\usepackage{graphicx}
\usepackage{multirow}
\usepackage{amsfonts}
\usepackage{amsmath}
\usepackage{amsthm}
\usepackage{amssymb}
\usepackage{accents}
\usepackage{indentfirst}
\usepackage{enumerate}
\usepackage{setspace}
\usepackage[noadjust]{cite}
\usepackage{xcolor}
\usepackage{float}

\usepackage{rotating}
\usepackage{pdflscape}
\usepackage{multicol}
\usepackage{afterpage}

\usepackage{xr-hyper}
\usepackage[colorlinks=true, allcolors=blue]{hyperref}

\def\A{{\mathcal A}}
\def\B{{\mathcal B}}
\def\Okubo{{\mathcal O}}

\def\R{{\mathbb R}}
\def\F{{\mathbb F}}

\def\N{{\mathbb N}}
\def\C{{\mathbb C}}

\def\rk{{\rm rk \,}}
\def\rank{{\rm rank}\,}
\def\chrs{{\rm char}\, }

\def\tr{{\rm tr}\,}

\def\Lin{{\mathcal L}\,}
\def\dim{{\rm dim}\,}

\makeatletter

\@addtoreset{equation}{section}

\newtheorem{theorem}{Theorem}[section]
\newtheorem{lemma}[theorem]{Lemma}
\newtheorem{corollary}[theorem]{Corollary}
\newtheorem{proposition}[theorem]{Proposition}

\theoremstyle{definition}
\newtheorem{example}[theorem]{Example}
\newtheorem{definition}[theorem]{Definition}

\theoremstyle{remark}
\newtheorem{remark}[theorem]{Remark}

\makeatother

\providecommand{\keywords}[1]{\textbf{Keywords:} #1}
\providecommand{\msc}[1]{\textbf{MSC 2020:} #1}

\begin{document}

\title{On the lengths of Okubo algebras}
\author{A.~E.~Guterman$^a$ and S.~A.~Zhilina$^{b,c,d,}$\thanks{The work of the second author was supported by the Ministry of Science and Higher Education of the Russian Federation (Goszadaniye No. 075-00337-20-03, project No. 0714-2020-0005).}}
\date{\small \em
$^a$Department of Mathematics, Bar-Ilan University, Ramat-Gan, 5290002, Israel\\
$^b$Department of Mathematics and Mechanics, Lomonosov Moscow State\\ University, Moscow, 119991, Russia\\
$^c$Moscow Center for Fundamental and Applied Mathematics, Moscow, 119991, Russia\\
$^d$Moscow Institute of Physics and Technology, Dolgoprudny, 141701, Russia
}

\maketitle

\begin{abstract}
We compute the lengths of two particular cases of (possibly non-unital) composition algebras, namely, standard composition algebras and Okubo algebras over an arbitrary field~$\F$. These results finish the complete description of lengths of symmetric composition algebras and finite-dimensional flexible composition algebras.
\end{abstract}


\msc{15A03, 17A20, 17A75.}

\let\thefootnote\relax\footnote{{\em Email addresses:} \texttt{alexander.guterman@biu.ac.il} (A.~E.~Guterman), \texttt{s.a.zhilina@gmail.com}\\ (S.~A.~Zhilina)}

\section{Introduction}

Let~$\A$ be an algebra over a field $\F$, not necessarily unital, not necessarily associative. We assume that~$\A$ is equipped with a strictly nondegenerate quadratic form~$n(\cdot)$ which means that the symmetric bilinear form $n(\cdot,\cdot)$ on~$\A$, associated to the norm $n(\cdot)$ and given by the formula $n(a,b) = n(a + b) - n(a) - n(b)$, is nondegenerate. This bilinear form is called the {\em polar form} of $n(\cdot)$. Clearly, $n(a,a) = 2n(a)$.

\begin{definition}
The algebra~$\A$ is called {\em a composition algebra} if the norm $n(\cdot)$ permits composition, that is, $n(ab) = n(a) n(b)$ for all $a, b \in \A$.
\end{definition}

In 1898 Hurwitz showed that the only possible values of dimension for a unital composition division algebra over $\R$ are $1$, $2$, $4$, and $8$. We refer to these algebras as to {\em classical Hurwitz algebras.} Later Hurwitz theorem was extended by Jacobson to arbitrary unital composition algebras over an arbitrary field~$\F$, $\chrs \F \ne 2$. Moreover, he showed that if $\chrs \F \neq 2$, then any unital composition algebra~$\A$ over~$\F$ is isomorphic to a Cayley--Dickson algebra $\A_n$ of dimension $2^n$, where $0 \le n \le 3$, see~\cite[p.~61, Theorem~1]{Jacobson}. This result was generalized in~\cite[p.~32, Theorem~1]{Zhevlakov} to a field~$\F$ of arbitrary characteristic. Given the unit element $e$ in~$\A$, the involution $a \mapsto \bar{a}$ on~$\A$ is defined as $\bar{a} = n(a,e) - a$, which satisfies $n(a) = a \bar{a} = \bar{a} a$ for all $a \in \A$. By using it, we can define standard composition algebras.

\begin{definition}[{\cite[p.~378]{ElduquePerez3}}] \label{definition:standard}
Let~$\A$ be a unital composition algebra over an arbitrary field~$\F$. Then a {\em standard composition algebra} is one of the following four algebras $\A^* = (\A,*)$ where
$$
\text{(I) } a*b = ab, \quad \text{(II) } a*b = \bar{a}b, \quad \text{(III) } a*b = a\bar{b}, \quad \text{(IV) } a*b = \bar{a}\bar{b}.
$$
\end{definition}

\begin{definition}
An algebra~$\A$ is called {\em flexible} if $(ab)a = a(ba)$ for all $a,b \in \A$.  
\end{definition}

It follows from the definition that unital composition algebras are exactly standard composition algebras of type I. The algebras of types II, III, and IV are non-unital. The algebras of types I and IV are flexible, and the algebras of types II and III are not. The algebras of type I are usually called {\em Hurwitz algebras,} and the algebras of type IV are called {\em para-Hurwitz algebras,} see~\cite{ElduqueMyung1}.

\medskip

In general, non-unital composition algebras can be rather complicated. In particular, by~\cite[Section~2]{ElduqueMyung2}, there exists a commutative (hence flexible) infinite-dimensional composition algebra over an arbitrary field~$\F$. Moreover, by~\cite[Examples~1 and~2]{ElduquePerez2}, over any field~$\F$ there is a left unital composition algebra of arbitrary infinite dimension. Hence some additional  properties are demanded to classify composition algebras, for example, finite dimensionality, power-associativity, flexibility, or division. It was shown in~\cite[p.~957]{Kaplansky} that, if the composition algebra~$\A$ is non-unital but there is a norm-one element $a \in \A$ such that the left and right multiplications by $a$, denoted respectively by $L_a$ and $R_a$, are bijections, then we can define a new multiplication on~$\A$ by
$$
x \times y = R_a^{-1}(x)L_a^{-1}(y).
$$
The algebra $(\A, \times)$ is a composition algebra with respect to the same quadratic form $n(\cdot)$, and $e = a^2$ is its unit element. Hence~$\A$ has dimension $1$, $2$, $4$, or~$8$. Such an element $a \in \A$ always exists if~$\A$ is finite-dimensional or~$\A$ is a division algebra.

It follows from~\cite[Theorem~3.2]{ElduqueMyung3} that any finite-dimensional flexible composition algebra~$\A$ over a field~$\F$ of arbitrary characteristic is either unital, i.e., a Hurwitz algebra, or symmetric, i.e., for any $a, b \in \A$ we have
$$
(a * b) * a = a * (b * a) = n(a) b.
$$
Full classification of symmetric composition algebras was obtained in numerous papers by Okubo and Osborn, and also by Elduque, Myung and P\'erez. It was shown in~\cite[Theorem~2.9]{Elduque5} that any symmetric composition algebra is either a form of a para-Hurwitz algebra or an Okubo algebra. Thus it is always finite-dimensional. By~\cite[p.~298]{Elduque1},~\cite[Theorems~4.2 and~4.3]{ElduqueMyung1}, and~\cite[Lemma~3.3]{ElduquePerez1}, a form of a para-Hurwitz algebra is either para-Hurwitz itself or isomorphic to a certain two-dimensional algebra. Okubo algebras, which are the main object of this paper, were first introduced and described in~\cite{Okubo1,Okubo2} over the fields of real and complex numbers. In~\cite{Okubo_symmetric1, Okubo_symmetric2} they were generalized to an arbitrary field~$\F$, $\chrs \F \neq 2$, and then the case when $\chrs \F = 2$ was considered in~\cite{ElduquePerez1, Elduque1}. We refer the reader to~\cite{Elduque2} for the survey on the properties of Okubo algebras and, in particular, their automorphisms and derivations.

\medskip

It is well known that various numerical invariants play an important role in the study of finite-dimensional algebras. One of them is the length function which describes the guaranteed number of multiplications which will be sufficient to generate the whole algebra with its arbitrary generating set. It was introduced in 1959 for the algebra of $3 \times 3$ matrices over an arbitrary field $\F$, see~\cite{Spencer-1959}. It was soon extended to square matrices of arbitrary size~\cite{Paz}, and then to arbitrary associative algebras. In particular, Pappacena~\cite{Pappacena} obtained the upper bound for the length of an arbitrary finite-dimensional unital associative algebra. The notion of length was generalized to unital non-associative algebras in~\cite{Guterman_upper-bounds, Guterman_sequences} where the method of characteristic sequences was introduced. This method provides a sharp upper bound for the length of a unital non-associative algebra, cf.~\cite[Section~4]{Guterman_upper-bounds}, a locally complex algebra, cf.~\cite[Section~6]{Guterman_upper-bounds}, and a quadratic algebra, see~\cite{Guterman_quadratic}. The lengths of classical Hurwitz algebras have been computed in~\cite{Guterman_Hurwitz-algebras}.

In our work we continue the study of the lengths in the non-associative case and focus on composition algebras. In~\cite{our_standard-composition-algebras} we have computed the lengths of standard composition algebras over an arbitrary field~$\F$ with $\chrs \F \neq 2$. The current paper aims to extend these results to the case of an arbitrary field~$\F$ and compute the lengths of another important class of composition algebras, namely, Okubo algebras. We will rely vastly on our previous paper~\cite{our_descendingly-flexible-algebras}, where we have introduced the conditions of descending flexibility and descending alternativity which impose restrictions on the lengths of $a(ba)$ and $(ab)a$, and of $a(ab)$ and $(ba)a$ for all $a,b \in \A$, respectively. In~\cite[Theorems~4.5 and~5.18]{our_descendingly-flexible-algebras} we have obtained an upper bound on the lengths of descendingly flexible and descendinly alternative algebras, and this result proves to be very useful for studying the lengths of composition algebras. The main tool used in its proof is the sequence of differences between the dimensions of certain nested linear subspaces. The properties of this sequence are of an independent interest, since, for instance, they appear to be extremely helpful in the proof of Lemma~\ref{lemma:two-element-generating} of the current paper, where we show that if $\Okubo$ is an Okubo algebra without nonzero idempotents and zero divisors, then the length of any its two-element generating system equals three.

\medskip

Our paper is organized as follows. In Subsection~\ref{subsection:length-definition} we provide a formal definition of the length function and some basic facts concerning it. In Subsection~\ref{subsection:mixing-and-sliding} we recall the definitions and properties of descendingly flexible and descendingly alternative algebras.

Section~\ref{section:standard-algebras} is devoted to standard composition algebras. In Lemma~\ref{lemma:standard-flexibility} of Subsection~\ref{subsection:identities} we show that standard composition algebras are both descendingly flexible and descendingly alternative. Based on the upper bound for the lengths of descendingly alternative algebras, in Subsection~\ref{subsection:standard-algebras} we present a new proof for the lengths of standard composition algebras over an arbitrary field $\F$, including the case when $\chrs \F = 2$. The main result of this section is Theorem~\ref{theorem:standard-length}.

We then consider Okubo algebras in Section~\ref{section:Okubo-algebra}. In Subsection~\ref{subsection:Okubo-properties} we provide a survey of Okubo algebras over arbitrary fields. All Okubo algebras are descindingly flexible, so it follows from the upper bound on the lengths of descendingly flexible algebras that the length of any Okubo algebra is at most four. In Theorem~\ref{theorem:Okubo-length} we show that, if an Okubo algebra has either nonzero idempotents or zero divisors, then this bound is achieved. Its proof uses the fact that such Okubo algebras have explicit and convenient multiplication tables, and by using them we can find a generating system whose length equals four. If an Okubo algebra~$\Okubo$ has no nonzero idempotents or zero divisors, then any its nonzero element generates a two-dimensional subalgebra, and any two elements, which do not belong to the same two-dimensional subalgebra, generate~$\Okubo$, see Lemma~\ref{lemma:Okubo-subalgebras} and Corollary~\ref{corollary:two-element-generating}. Hence, by Theorem~\ref{theorem:Okubo-length-exceptional}, the length of~$\Okubo$ equals three.

Thus the problem of length computation for standard composition algebras and Okubo algebras is solved completely. In view of~\cite[Theorem~2.9]{Elduque5},~\cite[Theorems~4.2 and~4.3]{ElduqueMyung1}, and~\cite[Theorem~3.2]{ElduqueMyung3}, Theorems~\ref{theorem:standard-length},~\ref{theorem:Okubo-length}, and~\ref{theorem:Okubo-length-exceptional} give us a complete description of lengths of symmetric composition algebras and finite-dimensional flexible composition algebras over an arbitrary field~$\F$, see Corollary~\ref{corollary:all-lengths}.

\section{Lengths of non-associative algebras} \label{section:length-properties}

\subsection{Definitions and main properties} \label{subsection:length-definition}

Let~$\A$ be an arbitrary finite-dimensional algebra over a field $\F$, possibly non-unital and non-associative. Given a subset $S \subseteq \A$ and $k \in \N$, any product of $k$ elements of $S$ is called a {\em word of length $k$} in letters from~$S$. If the algebra is unital, the unity~$e$ is considered as a word of {\em zero length} in $S$. Otherwise, we assume that there are no words of zero length. We denote the number of elements in $S$ by $|S|$, and the rank of $S$ as a linear system by $\rank(S)$.
 
Note that different arrangements of brackets provide different words of the same length. The set of all words in $S$ with length at most $k$ is denoted by $S^k$, $k \ge 0$. Similarly to the associative case, $k < m$ implies that $S^k \subseteq S^m$.

Recall that the linear span of $S$, denoted by $\Lin(S)$, is the set of all finite linear combinations of elements of $S$ with coefficients from~$\F$. We denote $\Lin_k(S) = \Lin(S^k)$, $k \geq 0$, and $\Lin_{\infty}(S) = \bigcup \limits_{k=0}^\infty \Lin_k(S)$. Clearly, for all $k$ we have $\Lin_k(S) \subseteq \Lin_{k+1}(S)$. Moreover, $\Lin_0(S) = \F$ if~$\A$ is unital, and $\Lin_0(S) = \{ 0 \}$ otherwise. It can be easily seen that a set $S$ is generating for~$\A$ if and only if $\A = \Lin_{\infty}(S)$.

\begin{definition} \label{definition:system-length} 
The	{\em length of a set} $S$ is $l(S) = \min \{ k \: | \: \Lin_k(S) = \Lin_{\infty}(S) \}$. 
\end{definition}

\begin{definition} \label{definition:algebra-length}
The	{\em length of an algebra~$\A$} is the maximum of lengths of its generating sets, i.e., $l(\A) = \max\limits_S \: \{l(S) \: | \: \Lin_{\infty}(S) = \A\}$. 
\end{definition}

If~$\A$ is associative, then $\Lin_{k}(S) = \Lin_{k+1}(S)$ implies $\Lin_{k}(S)=\Lin_{\infty}(S)$. Therefore, for a unital algebra~$\A$ we have $l(\A) \le \dim \A - 1$, and for a non-unital algebra~$\A$ we have $l(\A) \le \dim \A$. However, these inequalities do not hold for non-associative algebras, see~\cite[Example~2.2]{Guterman_Hurwitz-algebras}.

\begin{remark}[{\cite[Lemma~2.11, Corollary~2.12]{Guterman_upper-bounds}}] \label{remark:equal-linear-spans}
Let $S, S'$ be two subsets of~$\A$ such that $\Lin_1(S) = \Lin_1(S')$. Then $\Lin_k(S) = \Lin_k(S')$ for all $k \in \N$, so $\Lin_{\infty}(S) = \Lin_{\infty}(S')$. Hence $S$ is generating for~$\A$ if and only if $S'$ is generating for~$\A$, and $l(S) = l(S')$.
\end{remark}

We now define the sequence of differences for an arbitrary system $S$ which describes how fast the dimension of the linear span of words of length at most $k$ in $S$ is growing. It is a very useful tool for length computation, along with the method of characteristic sequences introduced in~\cite{Guterman_upper-bounds, Guterman_sequences}.

\begin{definition} \label{definition:sequence-differences}
Let $S \subseteq \A$. The sequence of differences between the dimensions of the linear spans of words in $S$ is $D(S) = \{ d_k(S) \}_{k = 0}^{\infty} = \{ d_k \}_{k = 0}^{\infty}$, where
\begin{align*}
    d_0 &= \dim \Lin_0(S) =
    \begin{cases}
    0, & \A \text{ is non-unital},\\
    1, & \A \text{ is unital};
    \end{cases}\\
    d_k &= \dim \Lin_k(S) - \dim \Lin_{k-1}(S), \quad k \in \N.
\end{align*}
Clearly, the definition of $d_0$ does not depend on a set $S$, and its value is a function of an algebra~$\A$ itself.
\end{definition}

\begin{proposition}[{\cite[Proposition~1.4]{our_descendingly-flexible-algebras}}] \label{proposition:differences-length}
Let~$\A$ be an arbitrary finite-dimensional algebra over a field $\F$, $S \subseteq \A$. Then
\begin{enumerate}[{\rm (1)}]
    \item $S$ is generating for~$\A$ if and only if $\sum \limits_{k = 0}^{\infty} d_k = \dim \A$;
    \item $l(S) = \max \{ k \in \N_0 \; | \; d_k \neq 0 \}$, where we assume that $\max \varnothing = 0$;
    \item $d_1 =
    \begin{cases}
    \rank(S), & \A \text{ is non-unital},\\
    \rank(S \cup \{ e \}) - 1, & \A \text{ is unital}.
    \end{cases}$
\end{enumerate}
\end{proposition}

\subsection{Descendingly flexible and descendingly alternative algebras} \label{subsection:mixing-and-sliding}

Consider $a, b, c \in \A$. We denote by $\Lin'_2(a,b,c)$ the linear span of all words of length at most two in $a,b,c$, except for the words $aa$, $bb$ and $cc$, that is,
$$
\Lin'_2(a,b,c) =
\begin{cases}
\phantom{e,{}}\Lin(a,b,c,ab,ba,cb,bc,ac,ca), & \A \text{ is non-unital},\\
\Lin(e,a,b,c,ab,ba,cb,bc,ac,ca), & \A \text{ is unital}.
\end{cases}
$$
Note also that
$$
\Lin_1(a,b,aa,ab,ba) = 
\begin{cases}
\phantom{e,{}}\Lin(a,b,aa,ab,ba), & \A \text{ is non-unital},\\
\Lin(e,a,b,aa,ab,ba), & \A \text{ is unital}.
\end{cases}
$$

\begin{definition} \label{definition:descending}
\leavevmode
\begin{itemize}
\item We say that an algebra~$\A$ is {\em descendingly flexible} if for all $a, b, c \in \A$ it holds that
\begin{align}
    (ab)a, \;\; a(ba) &\in \Lin_1(a,b,aa,ab,ba), \label{equation:general-flexibility}\\
    (ab)c + (cb)a &\in \Lin'_2(a,b,c), \label{equation:general-flexibility-1}\\
    a(bc) + c(ba) &\in \Lin'_2(a,b,c). \label{equation:general-flexibility-2}
\end{align}

\item We say that an algebra~$\A$ is {\em descendingly alternative} if for all $a, b, c \in \A$ it holds that
\begin{align}
    (ba)a, \;\; a(ab) &\in \Lin_1(a,b,aa,ab,ba), \label{equation:general-alternativity}\\
    (ab)c + (ac)b &\in \Lin'_2(a,b,c), \label{equation:general-alternativity-1}\\
    a(bc) + b(ac) &\in \Lin'_2(a,b,c). \label{equation:general-alternativity-2}
\end{align}
\end{itemize}
\end{definition}

Clearly, if $\chrs \F \neq 2$, then Eqs.~\eqref{equation:general-flexibility-1} and~\eqref{equation:general-flexibility-2} together imply Eq.~\eqref{equation:general-flexibility}. Similarly, Eqs.~\eqref{equation:general-alternativity-1} and~\eqref{equation:general-alternativity-2} imply Eq.~\eqref{equation:general-alternativity}. Therefore, conditions~\eqref{equation:general-flexibility} and~\eqref{equation:general-alternativity} are needed only for $\chrs \F = 2$.

The following proposition provides a convenient sufficient condition for an algebra to be descendingly flexible or descendingly alternative. It is a stronger form of Eqs.~\eqref{equation:general-flexibility} and~\eqref{equation:general-alternativity}.

\begin{proposition}[{\cite[Proposition~3.3]{our_descendingly-flexible-algebras}}] \label{proposition:sufficient-condition}
\leavevmode
\begin{enumerate}[{\rm (1)}]
    \item Assume that for all $a, b \in \A$ we have $(ab)a, \, a(ba) \in \Lin_1(a,b,aa,ab,ba)$, and the coefficient at $aa$ depends only on $b$. In other words, the expressions $(ab)a$ and $a(ba)$ can be represented as
    \begin{equation} \label{equation:general-flexibility-3}
    f_1(a,b)a + f_2(a,b)b + g(b)aa + f_3(a,b)ab + f_4(a,b)ba + f_5(a,b)e
    \end{equation}
    for some functions $f_j: \A \times \A \to \F$, $j = 1, \dots, 5$, and $g: \A \to \F$. If~$\A$ is non-unital, then we assume that $f_5(a,b) = 0$ for all $a,b \in \A$. Then~$\A$ is descendingly flexible.
    \item If for all $a, b \in \A$ it holds that $(ba)a, \, a(ab) \in \Lin_1(a,b,aa,ab,ba)$, and the coefficient at $aa$ depends only on $b$, then~$\A$ is descendingly alternative.
\end{enumerate}
\end{proposition}

\begin{example}
\leavevmode
\begin{enumerate}[{\rm (1)}]
    \item Any Okubo algebra $\Okubo$ over an arbitrary field~$\F$ is descendingly flexible, since it satisfies $(ab)a = a(ba) = n(a)b$ for all $a,b \in \Okubo$, see~\cite[p. 284, pp. 286-287]{Elduque1} and~\cite[p. 1199]{ElduqueMyung1}. However, Okubo algebras with nonzero idempotents or zero divisors are not descendingly alternative, see Example~\ref{example:Okubo-alternative} below.
    \item We prove in Lemma~\ref{lemma:standard-flexibility} that standard composition algebras are both descendingly flexible and descendingly alternative. In particular, this holds for any Cayley--Dickson algebra $\A_n$, $0 \leq n \leq 3$, over a field $\F$, $\chrs \F \neq 2$. However, in Proposition~\ref{proposition:Cayley-Dickson-flexibility} we show that for $n \geq 4$ the algebra $\A_n$ is neither descendingly flexible nor descendingly alternative whenever $\chrs \F \neq 2$.
\end{enumerate}
\end{example}

\begin{remark}
Descendingly flexible (alternative) algebras are not necessarily flexible (alternative). For example, standard composition algebras of types~II and~III are both descendingly flexible and descendingly alternative, see Lemma~\ref{lemma:standard-flexibility} below, however, they are neither flexible nor alternative. 

Conversely, flexible (alternative) algebras need not be descendingly flexible (alternative), see~\cite[Remark~3.5]{our_descendingly-flexible-algebras}. Moreover, \cite[Example~3.7]{our_descendingly-flexible-algebras} shows that the classes of descendingly flexible and descendingly alternative algebras are not contained in each other.
\end{remark}

We now list some of the most important results on the lengths of descendingly flexible and descendingly alternative algebras and their generating sets, which were obtained in~\cite{our_descendingly-flexible-algebras}.

\begin{lemma}[{\cite[Lemmas~2.5 and~2.8, Proposition~3.6]{our_descendingly-flexible-algebras}}] \label{lemma:full-stabilization} \label{lemma:general-recursive-length}
Assume that~$\A$ is either descendingly flexible or descendingly alternative, and let $S \subseteq \A$. Then
\begin{enumerate}[{\rm (1)}]
    \item for all $m \in \N$ we have
    $$
    \Lin_{m+1}(S) = \Lin_m(S) \cdot S + S \cdot \Lin_m(S) + \Lin_m(S);
    $$
    \item if $\Lin_m(S) = \Lin_{m+1}(S)$ for some $m \ge 0$, then $\Lin_m(S) = \Lin_{m+1}(S) = \Lin_{m+2}(S) = \dots$.
\end{enumerate}
\end{lemma}

\begin{theorem}[{\cite[Theorem~4.5]{our_descendingly-flexible-algebras}}] \label{theorem:descendingly-alternative-length}
Let~$\A$ be descendingly alternative, and $k = l(\A) \geq 2$. Then $\dim \A - d_0 \geq 2^{k-1} + k - 2$.
\end{theorem}

\begin{remark} \label{remark:descendingly-alternative-length}
For instance, in the cases when $k \in \{ 2, 3, 4 \}$, Theorem~\ref{theorem:descendingly-alternative-length} states that:
\begin{itemize}
    \item if $l(\A) \geq 2$, then $\dim \A - d_0 \geq 2$;
    \item if $l(\A) \geq 3$, then $\dim \A - d_0 \geq 5$;
    \item if $l(\A) \geq 4$, then $\dim \A - d_0 \geq 10$.
\end{itemize}
\end{remark}

\begin{proposition}[{\cite[Proposition~2.7 and Lemma~5.15]{our_descendingly-flexible-algebras}}] \label{proposition:differences-descendingly-flexible}
Let~$\A$ be descendingly flexible, $S \subseteq \A$. Then
\begin{enumerate}[{\rm (1)}]
    \item if $d_m = 0$ for some $m \in \N$, then $d_k = 0$ for all $k \geq m$;
    \item if $d_m \geq 1$ for some $m \in \{ 3, 4 \}$, then $d_k \geq 2$ for any $1 \leq k \leq m-1$.
\end{enumerate}
\end{proposition}

\begin{theorem}[{\cite[Theorem~5.18]{our_descendingly-flexible-algebras}}] \label{theorem:descendingly-flexible-length}
Let~$\A$ be descendingly flexible, and $k = l(\A)$. Then 
$$
\dim \A - d_0 \geq
\begin{cases}
k, & 1 \leq k \leq 2,\\
2k - 1, & 3 \leq k \leq 5,\\
3 \cdot 2^{k-4} + k - 3, & 6 \leq k.
\end{cases}
$$
\end{theorem}

\section{Standard composition algebras} \label{section:standard-algebras}

\subsection{Properties of standard composition algebras} \label{subsection:identities}

We refer the reader to~\cite[Chapter~2]{Zhevlakov} for a detailed description of Hurwitz algebras, i.e., unital composition algebras. Such algebras are classified by the following ``Generalized Hurwitz Theorem'' which uses the notion of the Cayley--Dickson construction. For more information on Cayley--Dickson algebras, see~\cite{our_split-algebras, Schafer} and references therein. Throughout this section,~$\A$ denotes a Hurwitz algebra over a field~$\F$ with the unit element~$e$.

\begin{theorem}[{\cite[p.~32, Theorem~1]{Zhevlakov}}] \label{theorem:classification}
Every Hurwitz algebra~$\A$ is isomorphic to one of the following types:
\begin{enumerate}[{\rm (1)}]
    \item The field~$\F$ for $\chrs \F \neq 2$ (otherwise, the norm $n(\cdot)$ is not strictly nondegenerate).
    \item The algebra $K(\mu) = \F + \F \ell_1$ with $\ell_1^2 = \ell_1 + \mu$, $4 \mu + 1 \neq 0$, and the involution $\overline{x + y \ell_1} = (x + y) - y \ell_1$ (if $\chrs \F \neq 2$, then $K(\mu)$ is isomorphic to the algebra $(\F,\alpha)$ which is obtained from~$\F$ by the Cayley--Dickson process with the parameter $\alpha = (4\mu + 1)/4 \neq 0$).
    \item A generalized quaternion algebra $Q(\mu, \beta) = (K(\mu), \beta) = K(\mu) + K(\mu) \ell_2$ which is obtained from $K(\mu)$ by the Cayley--Dickson process with the parameter $\beta \neq 0$.
    \item A generalized octonion algebra $C(\mu, \beta, \gamma) = (Q(\mu, \beta), \gamma) = Q(\mu, \beta) + Q(\mu, \beta) \ell_4$ which is obtained from $Q(\mu, \beta)$ by the Cayley--Dickson process with the parameter $\gamma \neq 0$.
\end{enumerate}
\end{theorem}

Given $a,b,c \in \A$, we denote their associator by $[a,b,c] = (ab)c - a(bc)$. It is well known that~$\A$ is alternative, that is, $[a,a,b] = [b,a,a] = 0$ for all $a,b \in \A$, cf.~\cite[p.~25, Lemma~1]{Zhevlakov}. Hence~$\A$ is flexible, that is, $[a,b,a] = 0$ for all $a,b \in \A$, see~\cite[p.~35]{Zhevlakov}. 

The {\em trace} of $a \in \A$ is given by $t(a) = n(a,e) \in \F$. Then the involution $a \mapsto \bar{a}$ on~$\A$ is defined as $\bar{a} = t(a) - a$. By~\cite[p.~26, Lemma~2]{Zhevlakov}, it satisfies $a\bar{a} = \bar{a}a = n(a)$, that is, the involution is regular. The main properties we will use in this section are regularity of involution and alternativity of~$\A$.

\begin{proposition}[{\cite[p.~26]{Zhevlakov}}] \label{proposition:quadratic}
For any $a \in \mathcal{A}$ we have $a^2 - t(a)a + n(a) = 0$.
\end{proposition}

\begin{proposition}[{\cite[p.~27]{Zhevlakov}}] \label{proposition:scalar-product}
For any $a, b \in \A$ we have $n(a,b) = a\bar{b} + b\bar{a} = \bar{a}b + \bar{b}a \in \F$.
\end{proposition}

We denote by $\A^* = (\A,*)$ a standard composition algebra of type I, II, III, or IV, where
$$
\text{(I) } a*b = ab, \quad \text{(II) } a*b = \bar{a}b, \quad \text{(III) } a*b = a\bar{b}, \quad \text{(IV) } a*b = \bar{a}\bar{b}.
$$
The algebra $\A^*$ of type II is anti-isomorphic to the algebra $\A^*$ of type III under the mapping $a \to \bar{a}$, cf.~\cite[Theorem~3.5, Theorem~4.5]{Beites}. Hence their lengths coincide, so it is sufficient for us to consider algebras of types I, II and IV.

\begin{lemma} \label{lemma:two-product-simplification}
Let $a,b \in \A^*$. Then $a*b + b*a \in \Lin(e,a,b)$.
\end{lemma}

\begin{proof}
Consider three cases, depending on the type of $\A^*$.
\begin{itemize}
\item[\rm I:]
By Proposition~\ref{proposition:scalar-product}, $n(a,b) = a\bar{b} + b\bar{a} = a(t(b) - b) + b(t(a) - a) = t(b)a + t(a)b - (ab + ba)$. Hence $ab + ba = t(b)a + t(a)b - n(a,b) \in \Lin(e,a,b)$.

\item[\rm II:]
Proposition~\ref{proposition:scalar-product} implies that $a*b + b*a = \bar{a}b + \bar{b}a = n(a,b) \in \Lin(e)$.

\item[\rm IV:]
We have
\begin{align*}
a*b + b*a &= \bar{a}\bar{b} + \bar{b}\bar{a} = (t(a) - a)\bar{b} + (t(b) - b)\bar{a} \\
&= t(a)\bar{b} + t(b)\bar{a} - (a\bar{b} + b\bar{a}) \\
&= t(a)(t(b) - b) + t(b)(t(a) - a) - n(a,b) \\
&= (2t(a)t(b) - n(a,b)) - t(a)b - t(b)a. \qedhere
\end{align*}
\end{itemize}
\end{proof}

Combined with Proposition~\ref{proposition:sufficient-condition}, the following lemma shows that $\A^*$ satisfies both the conditions of descending flexibility and descending alternativity.

\begin{lemma} \label{lemma:standard-flexibility}
Let $a, b \in \A^*$. Then $(a*b)*a, \: a*(b*a), \: (b*a)*a, \: a*(a*b) \in \Lin(a,b,a*a,a*b,b*a)$, and the coefficient at $a*a$ depends only on $b$.
\end{lemma}

\begin{proof}
We consider three cases, depending on the type of $\A^*$.
\begin{itemize}
\item[\rm I:]
$
\begin{aligned}[t]
(ab)a &= (n(a,\bar{b}) - \bar{b}\bar{a})a = n(a,\bar{b})a - (t(b) - b)(\bar{a}a) \\
&= n(a,\bar{b})a - t(b)(t(a) - a)a + n(a)b \\
&= (n(a,\bar{b}) - t(a)t(b)) a + t(b)aa + n(a)b;\\
(ba)a &= b(aa) = b(t(a)a - n(a)) = t(a) ba - n(a)b;\\
a(ab) &= (aa)b = (t(a)a - n(a))b = t(a) ab - n(a)b.
\end{aligned}
$

Moreover, by flexibility, $a(ba) = (ab)a$.

\item[\rm II:]
$
\begin{aligned}[t]
(a*b)*a &= (\bar{a}b)*a = \overline{(\bar{a}b)}a = (\bar{b}a)a \\
&= (\bar{b}a)(t(a) - \bar{a}) = t(a)\bar{b}a - \bar{b}(a\bar{a}) \\
&= t(a)\bar{b}a - n(a)(t(b) - b) \\
&= t(a)b*a + n(a)b - t(b)a*a;\\
a*(b*a) &= \bar{a}(\bar{b}a) = (t(a) - a)(\bar{b}a) \\
&= t(a)\bar{b}a - a(n(a,b) - \bar{a}b) \\
&= t(a)\bar{b}a - n(a,b)a + (a\bar{a})b \\
&= t(a) b*a - n(a,b)a + n(a)b;
\end{aligned}
$\\
$
\begin{aligned}[t]
(b*a)*a &= (n(a,b)e - a*b)*a \\
&= n(a,b)e*a - (a*b)*a \\
&= n(a,b)a - t(a)b*a - n(a)b + t(b)a*a;\\
a*(a*b) &= \bar{a}(\bar{a}b) = (\bar{a}\bar{a})b \\
&= (t(\bar{a})\bar{a} - n(\bar{a})) b = (t(a)\bar{a} - n(a)) b \\
&= t(a) \bar{a}b - n(a)b = t(a) a*b - n(a)b.
\end{aligned}
$

\item[\rm IV:]
$
\begin{aligned}[t]
    (a*b)*a &= (\bar{a}\bar{b})*a = \overline{(\bar{a}\bar{b})}\bar{a} = (ba)\bar{a} = b(a\bar{a}) = n(a)b;\\
    a*(b*a) &= a*(\bar{b}\bar{a}) = \bar{a}\overline{(\bar{b}\bar{a})} = \bar{a}(ab) = (\bar{a}a)b = n(a)b;\\
    (b*a)*a &= ((2t(a)t(b) - n(a,b))e - t(a)b - t(b)a - a*b)*a \\
    &= (2t(a)t(b) - n(a,b)) e*a - t(a)b*a - t(b)a*a - (a*b)*a \\
    &= (2t(a)t(b) - n(a,b)) (t(a) - a) - t(a)b*a - t(b)a*a - n(a) b \\
    &= t(a) ((2t(a)t(b) - n(a,b)) - b*a) \\
    &+ (n(a,b) - 2t(a)t(b))a - t(b)a*a - n(a) b \\
    &= t(a) (a*b + t(a)b + t(b) a) + (n(a,b) - 2t(a)t(b))a - t(b)a*a - n(a) b \\
    &= t(a) a*b + ((t(a))^2 - n(a)) b + (n(a,b) - t(a)t(b)) a - t(b)a*a;\\
    a*(a*b) &= t(a) b*a + ((t(a))^2 - n(a)) b + (n(a,b) - t(a)t(b)) a - t(b)a*a.
\end{aligned}
$
\end{itemize}
\end{proof}

If $\chrs \F \neq 2$, then any Hurwitz algebra~$\A$ over~$\F$ is isomorphic to a Cayley--Dickson algebra $\A_n$ of dimension $2^n$, where $0 \le n \le 3$, and vice versa. Hence it follows from Lemma~\ref{lemma:standard-flexibility} that any Cayley--Dickson algebra $\A_n$, $0 \leq n \leq 3$, over an arbitrary field $\F$, $\chrs \F \neq 2$, is both descendingly flexible and descendingly alternative, see Definition~\ref{definition:descending}. We now show that this statement does not hold for higher-dimensional Cayley--Dickson algebras.

\begin{proposition} \label{proposition:Cayley-Dickson-flexibility}
If $n \geq 4$, then a Cayley--Dickson algebra $\A_n$ over a field~$\F$ with $\chrs \F \neq 2$ is neither descendingly flexible nor descendingly alternative.
\end{proposition}

\begin{proof}
Consider the standard basis $\{ e {\:=\:} e_0, e_1, \dots, e_{2^n-1} \}$ of $\A_n$, and let $\{ \gamma_0, \gamma_1, \dots, \gamma_{n-1} \}$ be the sequence of (nonzero) Cayley--Dickson parameters of $\A_n$. We set $a = e_1 + e_{10}$ and $b = e_3 + e_{15}$. Then
\begin{align*}
    aa &{}= (\gamma_0 - \gamma_1 \gamma_3) e_0;\\
    ab = -ba &{}= \gamma_0 e_2 + \gamma_1 \gamma_3 e_5 + \gamma_1 e_9 + \gamma_0 e_{14}\\
    (ab)a = a(ba) = -(ba)a = -a(ab) &{}= (\gamma_1 \gamma_3 - \gamma_0)e_3 - 2 \gamma_0 \gamma_1 \gamma_3 e_4 + (\gamma_1 \gamma_3 - \gamma_0)e_{15}.
\end{align*}
Since $\chrs \F \neq 2$ and $\gamma_j \neq 0$ for all $j \in \{0, \dots, 2^n-1 \}$, the coefficient at $e_4$ in $(ab)a$ is nonzero. Hence $(ab)a = a(ba) = -(ba)a = -a(ab) \notin \Lin(e,a,b,aa,ab,ba)$, so $\A_n$ is neither descendingly flexible nor descendingly alternative.
\end{proof}

\subsection{Lengths of standard composition algebras} \label{subsection:standard-algebras}

If $\A^*$ is of type I or $\dim \A = 1$, then $\A^* = \A$ is unital, and in this case $\Lin_0(S) = \F$ for any subset $S \subseteq \A^*$. Hence, if $\A^*$ is of type II, III or IV, it is sufficient to consider only $\dim \A \geq 2$. Then $\A^*$ becomes non-unital, so $\Lin_0(S) = \{ 0 \}$ for any $S$.

\begin{proposition} \label{proposition:standard-length}
\leavevmode
\begin{enumerate}[{\rm (1)}]
    \item If $\A^*$ is of type I, then $l(\A^*) \ge \log_2 (\dim \A^*)$.
    \item If $\F = \F_2$, $\A = K(0)$, and $\A^*$ is of type II or III, then $l(\A^*) = 1$.
    \item If $\F = \F_2$, $\A = K(1)$, and $\A^*$ is of type IV, then $l(\A^*) = 1$.
    \item Otherwise, $l(\A^*) \ge \max \{ 2, \, \log_2 (\dim \A^*) \}$. Clearly, this value differs from $\log_2 (\dim \A^*)$ for $\dim \A^* = 2$ only.
\end{enumerate}
\end{proposition}
	
\begin{proof}
We use here the classification of Hurwitz algebras from Theorem~\ref{theorem:classification}. Let $\dim \A = 2^n$. If $\A^*$ is of type I, then we set $S = \{ \ell_{2^k} \: \vert \: 0 \leq k \leq n - 1 \}$. It can be easily seen that $S$ is generating for~$\A^*$ and $l(S) = n = \log_2 (\dim \A^*)$.

If $\A^*$ is of type II or IV and $2 \leq n \leq 3$, then $S$ as above is again generating for~$\A^*$, $e \in \Lin(\ell_2, \ell_2*\ell_2)$ is a word of length $2$ in $S$, and $l(S) = \log_2 (\dim \A^*)$.

Assume that $\A^*$ is of type II or IV with $n = 1$. Then $\A = K(\mu)$, i.e., $\A = \F + \F \ell_1$ with $\ell_1^2 = \ell_1 + \mu$. Let $S$ be an arbitrary generating system of $\A^*$. By Remark~\ref{remark:equal-linear-spans}, we may assume that $S$ is linearly independent. If $|S| = 2$, then we have $\Lin_1(S) = \A^*$, so $l(S) \leq 1$. Assume now that $|S| = 1$. Then $S = \{ a \}$ for some $a = x + \ell_1$, $x \in \F$. We have 
\begin{align*}
    a^2 &{}= (x + \ell_1)^2 = x^2 + 2 x \ell_1 + \ell_1^2 \\
    &{}= x^2 + \mu + (2 x + 1) \ell_1 = \mu - x^2 - x + (2 x + 1) a,\\
    \bar{a} &{}= (x + 1) - \ell_1 = 2 x + 1 - a.
\end{align*}
\begin{itemize}
    \item If $\A^*$ is of type II, then 
    $$
    a*a = \bar{a}a = (2 x + 1 - a)a = (2 x + 1)a - a^2 = x^2 + x - \mu.
    $$
    We have $x^2 + x - \mu = 0$ for all $x \in \F$ if and only if $\mu = 0$ and $\F = \F_2$. In this case $\Lin_2(S) = \Lin(S)$ for any $S \subseteq \A^*$ such that $|S| = 1$. Hence any generating system~$S$ for~$\A^*$ must contain at least two linearly independent elements, and thus $l(\A^*) = 1$.
    
    Otherwise, there is some $x \in \F$ such that $x^2 + x - \mu \neq 0$. Then $e \in \Lin_2(S)$, so $S = \{ x + \ell_1 \}$ is generating for~$\A$, and $l(S) = 2$.
    
    \item If $\A^*$ is of type IV, then 
    \begin{align*}
    a*a &{}= \bar{a}\bar{a} = (2 x + 1 - a)(2 x + 1 - a) = (2 x + 1)^2 - 2(2 x + 1)a + a^2 \\
    &{}= (2 x + 1)^2 + \mu - x^2 - x - (2 x + 1) a = 3x^2 + 3x + 1 + \mu - (2 x + 1) a.
    \end{align*}    
    Under the condition that $4 \mu + 1 \neq 0$, we have $3x^2 + 3x + 1 + \mu = 0$ for all $x \in \F$ if and only if $\mu = 1$ and $\F = \F_2$. In this case $\Lin_2(S) = \Lin(S)$ for any $S \subseteq \A^*$ such that $|S| = 1$. Hence any generating system $S$ for $\A^*$ must contain at least two linearly independent elements, and thus $l(\A^*) = 1$.
    
    Otherwise, there is some $x \in \F$ such that $3x^2 + 3x + 1 + \mu \neq 0$. Then $e \in \Lin_2(S)$, so $S = \{ x + \ell_1 \}$ is generating for~$\A$, and $l(S) = 2$. \qedhere
\end{itemize}
\end{proof}

\begin{theorem} \label{theorem:standard-length}
\leavevmode
\begin{enumerate}[{\rm (1)}]
    \item If $\A^*$ is of type I, then $l(\A^*) = \log_2 (\dim \A^*)$.
    \item If $\F = \F_2$, $\A = K(0)$, and $\A^*$ is of type II or III, then $l(\A^*) = 1$.
    \item If $\F = \F_2$, $\A = K(1)$, and $\A^*$ is of type IV, then $l(\A^*) = 1$.
    \item Otherwise, $l(\A^*) = \max \{ 2, \, \log_2 (\dim \A^*) \}$.
\end{enumerate}
\end{theorem}

\begin{proof}
By Proposition~\ref{proposition:standard-length}, it is sufficient to prove the upper bounds in (1) and (4). By Proposition~\ref{proposition:sufficient-condition} and Lemma~\ref{lemma:standard-flexibility}, $\A^*$ is descendingly alternative, so we can use Theorem~\ref{theorem:descendingly-alternative-length}. For all values of $\dim \A^* \in \{1, 2, 4, 8\}$, the exact upper bound for $l(\A^*)$ follows from Theorem~\ref{theorem:descendingly-alternative-length} both in the unital and in the non-unital cases.
\end{proof}

\section{Okubo algebras} \label{section:Okubo-algebra}

\subsection{Definition and main properties} \label{subsection:Okubo-properties}

We use~\cite[p.~1198]{ElduqueMyung1} to define Okubo algebras over an arbitrary field $\F$, $\chrs \F \neq 2,3$. Assume first that~$\F$ is algebraically closed. Then~$\F$ contains the solution $\mu = \frac{3 \pm \sqrt{-3}}{6}$ to the equation $3X(1-X) = 1$. Let $sl_3(\F)$ denote the Lie algebra of $3 \times 3$ trace zero matrices over~$\F$. We define a new non-associative product ``$*$'' on $sl_3(\F)$ by
$$
x * y = \mu xy + (1-\mu) yx - \frac{\tr(xy)}{3}I,
$$
where $xy$ denotes the matrix product, and $I \in M_3(\F)$ is the identity matrix. The resulting algebra, denoted by $P_8(\F)$, is called the {\em pseudo-octonion algebra} over~$\F$. The quadratic form on $P_8(\F)$ is introduced by
$$
n(x) = \frac{\tr(x^2)}{6}.
$$
Then $P_8(\F)$ is a flexible composition algebra over~$\F$. Moreover, it is a symmetric composition algebra, see~\cite[pp.~42--43]{Okubo}. Recall that a composition algebra $(\A, *)$ is called {\em symmetric} if for any $x, y \in \A$ we have
\begin{equation} \label{equation:Okubo-flexibility}
(x * y) * x = x * (y * x) = n(x) y,
\end{equation}
 By linearizing Eq.~\eqref{equation:Okubo-flexibility}, we obtain that
\begin{equation} \label{equation:Okubo-flexibility-linearized}
(x * y) * z + (z * y) * x = x * (y * z) + z * (y * x) = n(x,z) y.
\end{equation}

Assume now that~$\F$ is not algebraically closed, and let~$\overline{\F}$ be an algebraic closure of~$\F$. An algebra~$\A$ over~$\F$ is called an $\F$-form of an algebra~$\B$ over~$\overline{\F}$ if its scalar extension to $\overline{\F}$ is isomorphic to~$\B$, that is, $\A \otimes_{\F} \overline{\F} \cong \B$. Then an Okubo algebra over~$\F$, $\chrs \F \neq 2, 3$, is defined as an arbitrary $\F$-form of $P_8(\overline{\F})$.

\begin{example}
The first Okubo algebras to be constructed were the complex algebra $P_8(\C)$ and its real form $\widetilde{P}_8(\R)$ which is called the {\em real pseudo-octonion algebra} and defined by
$$
\widetilde{P}_8(\R) = \{ x  \in M_3(\C) \; | \; x \text{ is Hermitian and } \tr(x) = 0 \},
$$
see~\cite{Okubo1} and~\cite[Section~4.2]{Okubo}. The algebra $P_8(\C)$ is not a division algebra, since it contains elements of zero norm, however, this is not the case for its real form $\widetilde{P}_8(\R)$. The algebra $\widetilde{P}_8(\R)$ is a flexible division composition algebra over $\R$.
\end{example}

If $\chrs \F = 3$, then one has to use another approach to define Okubo algebras and derive their properties, see~\cite{Okubo_symmetric2}. The behavior of Okubo algebras over a field~$\F$ with $\chrs \F = 2$ is similar to the general case and can be found in~\cite{ElduquePerez1, Elduque1}. Most importantly, any Okubo algebra over an arbitrary field~$\F$ is a symmetric composition algebra, i.e., satisfies~Eq.~\eqref{equation:Okubo-flexibility}. We denote an arbitrary Okubo algebra by $\Okubo$.

Recall that our definition of composition algebras assumes that the norm $n(\cdot)$ is strictly nondegenerate, that is, the corresponding symmetric bilinear form $n(a,b) = n(a+b) - n(a) - n(b)$ is nondegenerate. By~\cite[Lemma~2.3]{Okubo_symmetric2} and~\cite[(34.1)]{Knus}, an algebra $(\A,*)$ over an arbitrary field~$\F$ with a strictly nondegenerate quadratic form $n(\cdot)$ satisfies Eq.~\eqref{equation:Okubo-flexibility} if and only if it is a composition algebra with {\em associative bilinear form,} that is,
\begin{equation} \label{equation:switch-scalar-product}
n(x*y,z) = n(x,y*z)
\end{equation}
for all $x,y,z \in \A$.

\begin{proposition}
Symmetric composition algebras are descendingly flexible.
\end{proposition}

\begin{proof}
Follows immediately from Eq.~\eqref{equation:Okubo-flexibility} and Proposition~\ref{proposition:sufficient-condition}.
\end{proof}

By~\cite[Lemma~3]{Okubo_symmetric1}, another important example of symmetric composition algebras are para-Hurwitz algebras, see Definition~\ref{definition:standard}. The element $e$ in a para-Hurwitz algebra $\A^*$ is a {\em para-unit}, that is, it satisfies $e*a = a*e = n(e,a) a - a = \bar{a}$ for all $a \in \A^*$. Besides, $e*e = \bar{e} = e$, so $e$ is a nonzero idempotent in $\A^*$.

\begin{theorem}[{\cite[p.~298]{Elduque1}, \cite[Theorems~4.2 and~4.3]{ElduqueMyung1}, \cite[Lemma~3.3]{ElduquePerez1}}] \label{theorem:forms-of-para-Hurwitz}
An algebra~$\A$ over~$\F$ is an $\F$-form of a para-Hurwitz algebra over~$\overline{\F}$ if and only if one of the following condition holds:
\begin{enumerate}[{\rm (1)}]
    \item $\A$ is a para-Hurwitz algebra,
    \item $\dim \A = 2$ and there exists a basis $\{ u, v \}$ in~$\A$ which satisfies $u^2 = v$, $uv = vu = u$ and $v^2 = \lambda u - v$ for some $\lambda \in \F$ such that the polynomial $x^3 - 3x - \lambda$ is irreducible over $\F$.
\end{enumerate}
\end{theorem}

The canonical multiplication table of Okubo algebras over an algebraically closed field~$\F$ of arbitrary characteristic is given by~\cite[p.~4, Table~1]{Elduque2} which coincides with Table~\ref{table:okubo-algebra-isotropic} for $\alpha = \beta = 1$. Clearly, if~$\F$ is algebraically closed, then the norm on any Okubo algebra over~$\F$ is isotropic, i.e., there exists nonzero $x \in \Okubo$ such that $n(x) = 0$. As the following theorem shows, this is also true for some other fields which need not be algebraically closed.

\begin{theorem}[{\cite[Corollary~3.4]{Elduque1}, \cite[Corollary~5]{Elduque3}, \cite[Proposition~7.3]{ElduqueMyung2}, \cite[Lemma~3.7]{ElduquePerez1}}] \label{theorem:Okubo-algebras-isotropic}
Assume that one of the following condition holds:
\begin{enumerate}[{\rm (1)}]
    \item $\chrs \F = 3$,
    \item $\chrs \F \neq 3$, and~$\F$ contains the cubic roots of $1$.
\end{enumerate}
Then the norm on any Okubo algebra over~$\F$ is isotropic.
\end{theorem}

\begin{theorem}[{\cite[Theorem~7]{Elduque3}, \cite[Theorem~4.2]{Elduque6}}] \label{theorem:Okubo-isotropic-norm-classification}
An Okubo algebra over a field~$\F$ has isotropic norm if and only if its multiplication table is given by Table~\ref{table:okubo-algebra-isotropic} for some $\alpha, \beta \in \F \setminus \{ 0 \}$.
\end{theorem}

In the case when an Okubo algebra with isotropic norm has nonzero idempotents, its multiplication table can be simplified even more.

\begin{theorem}[{\cite[Theorem~3.18]{Elduque4} \cite[Theorem~5.9]{Elduque5}}]
The Okubo algebras over an arbitrary field~$\F$ with isotropic norm and nonzero idempotents are precisely the algebras
with multiplication Table~\ref{table:okubo-algebra-isotropic} for
\begin{enumerate}[{\rm (1)}]
    \item $\alpha = \beta = 1$ if $\chrs \F \neq 3$;
    \item $\alpha = 1$ and some $\beta \in \F \setminus \{ 0 \}$ if $\chrs \F = 3$.
\end{enumerate}
\end{theorem}

\begin{theorem}[{\cite[Theorem~3.2]{Elduque6}}] \label{theorem:Okubo-algebras-finite}
If~$\F$ is finite, then any Okubo algebra over~$\F$ has multiplication Table~\ref{table:okubo-algebra-isotropic} with $\alpha = \beta = 1$, and thus has isotropic norm and nonzero idempotents.
\end{theorem}

The following theorem of~\cite{ElduquePerez1} provides a complete classification of symmetric composition algebras with nonzero idempotents. In the case when $\chrs \F \neq 3$, an Okubo algebra is obtained by choosing a new multiplication on the Hurwitz algebra $C(-1,\beta,\gamma)$, where the notation from Theorem~\ref{theorem:classification} is used. However, we only need the multiplication table of the resulting algebra, so we omit the detailed description of this construction.

\begin{theorem}[{\cite[Theorems~A$'$ and B$'$]{ElduquePerez1}}] \label{theorem:symmetric-idempotents-classification}
An algebra~$\A$ over a field~$\F$ is a symmetric composition algebra with nonzero idempotents if and only if either~$\A$ is para-Hurwitz or~$\A$ is an Okubo algebra with multiplication table given by
\begin{enumerate}[{\rm (1)}]
    \item Table~\ref{table:okubo-algebra-idempotents} for some $\beta, \gamma \in \F \setminus \{ 0 \}$ if $\chrs \F \neq 3$;
    \item Table~\ref{table:okubo-algebra-isotropic} for $\alpha = 1$ and some $\beta \in \F \setminus \{ 0 \}$ if $\chrs \F = 3$.
\end{enumerate}
\end{theorem}

\afterpage{
\begin{landscape}
\begin{table}[H]
\centering
$
\begin{array}{|c||c|c|c|c|c|c|c|c|}
\hline
\vphantom{\Big|} * & x_{1,0} & x_{-1,0} & x_{0,1} & x_{0,-1} & x_{1,1} & x_{-1,-1} & x_{-1,1} & x_{1,-1} \\\hline\hline
\vphantom{\Big|} x_{1,0} & -\alpha x_{-1,0} & 0 & 0 & x_{1,-1} & 0 & x_{0,-1} & 0 & \alpha x_{-1,-1}\\\hline
\vphantom{\Big|} x_{-1,0} & 0 & -\alpha^{-1} x_{1,0} & x_{-1,1} & 0 & x_{0,1} & 0 & \alpha^{-1} x_{1,1} & 0\\\hline
\vphantom{\Big|} x_{0,1} & x_{1,1} & 0 & -\beta x_{0,-1} & 0 & \beta x_{1,-1} & 0 & 0 & x_{1,0}\\\hline
\vphantom{\Big|} x_{0,-1} & 0 & x_{-1,-1} & 0 & -\beta^{-1} x_{0,1} & 0 & \beta^{-1} x_{-1,1} & x_{-1,0} & 0\\\hline
\vphantom{\Big|} x_{1,1} & \alpha x_{-1,1} & 0 & 0 & x_{1,0} & -(\alpha \beta) x_{-1,-1} & 0 & \beta x_{0,-1} & 0\\\hline
\vphantom{\Big|} x_{-1,-1} & 0 & \alpha^{-1} x_{1,-1} & x_{-1,0} & 0 & 0 & -(\alpha \beta)^{-1} x_{1,1} & 0 & \beta^{-1} x_{0,1}\\\hline
\vphantom{\Big|} x_{-1,1} & x_{0,1} & 0 & \beta x_{-1,-1} & 0 & 0 & \alpha^{-1} x_{1,0} & -\alpha^{-1} \beta x_{1,-1} & 0\\\hline
\vphantom{\Big|} x_{1,-1} & 0 & x_{0,-1} & 0 & \beta^{-1} x_{1,1} & \alpha x_{-1,0} & 0 & 0 & -\alpha \beta^{-1} x_{-1,1}\\\hline
\end{array}
$
\caption{\label{table:okubo-algebra-isotropic} Multiplication table of Okubo algebra with isotropic norm.}
\end{table}

\begin{table}[H]
\centering
$
\begin{array}{|c||c|c|c|c|c|c|c|c|}
\hline
\vphantom{\Big|} * & x_0 & x_1 & x_2 & x_3 & x_4 & x_5 & x_6 & x_7 \\\hline\hline
\vphantom{\Big|} x_0 & x_0 & -x_0 - x_1 & -x_2 & -x_3 & x_4 + x_5 & -x_4 & x_6 + x_7 & -x_6\\\hline
\vphantom{\Big|} x_1 & -x_0 - x_1 & x_1 & x_2 + x_3 & -x_2 & -x_5 & x_4 + x_5 & -x_6 & -x_7\\\hline
\vphantom{\Big|} x_2 & -x_2 & -x_3 & \beta x_0 & -\beta (x_0 + x_1) & -x_6 - x_7 & x_6 & -\beta (x_4 + x_5) & \beta x_4\\\hline
\vphantom{\Big|} x_3 & -x_3 & x_2 + x_3 & \beta x_1 & \beta x_0 & x_6 & x_7 & \beta x_5 & -\beta (x_4 + x_5)\\\hline
\vphantom{\Big|} x_4 & -x_5 & x_4 + x_5 & -x_7 & x_6 + x_7 & -\gamma (x_0 + x_1) & \gamma x_1 & -\gamma x_3 & \gamma (x_2 + x_3)\\\hline
\vphantom{\Big|} x_5 & x_4 + x_5 & -x_4 & x_6 + x_7 & -x_6 & \gamma x_0 & -\gamma (x_0 + x_1) & -\gamma x_2 & -\gamma x_3\\\hline
\vphantom{\Big|} x_6 & -x_7 & -x_6 & -\beta x_5 & -\beta x_4 & -\gamma (x_2 + x_3) & \gamma x_3 & -\beta \gamma x_1 & \beta \gamma (x_0 + x_1)\\\hline
\vphantom{\Big|} x_7 & x_6 + x_7 & -x_7 & \beta (x_4 + x_5) & -\beta x_5 & \gamma x_2 & -\gamma (x_2 + x_3) & -\beta \gamma x_0 & -\beta \gamma x_1\\\hline
\end{array}
$
\caption{\label{table:okubo-algebra-idempotents} Multiplication table of Okubo algebra with nonzero idempotents, $\chrs \F \neq 3$.}
\end{table}
\end{landscape}
}

By~\cite[(34.10)]{Knus}, if~$\A$ is a symmetric composition algebra over a field $\F$, then either~$\A$ contains an idempotent, or there is a cubic field extension $\mathbb{K}/\F$ such that $\A \otimes_{\F} \mathbb{K}$ contains an idempotent. This result has two immediate corollaries.

\begin{proposition}[{\cite[Corollary~2]{Okubo_symmetric1}}] \label{proposition:contains-idempotent}
Let~$\F$ be either algebraically or real closed field. If~$\A$ is a symmetric composition algebra over $\F$, then~$\A$ contains a nonzero idempotent.
\end{proposition}

\begin{corollary}[{\cite[Theorem~2.9]{Elduque5}}] \label{corollary:symmetric-composition-classification}
Any symmetric composition algebra~$\A$ over an arbitrary field~$\F$ is either a form of a para-Hurwitz algebra (see Theorem~\ref{theorem:forms-of-para-Hurwitz}) or an Okubo algebra.
\end{corollary}

The following theorem was originally proved by Okubo~\cite[Theorem~1.1]{Okubo_flexible1} in the case when $\chrs \F \neq 2, 3$. Then Elduque and Myung gave another proof of this fact in~\cite[Theorem~1.2]{ElduqueMyung1} and classified Okubo algebras over an arbitrary field $\F$, $\chrs \F \neq 2, 3$, in~\cite[Theorem~6.2]{ElduqueMyung1}. These results were later improved in~\cite[Main Theorem; Sections~5 and~6]{ElduqueMyung2}, where a relationship between flexible composition algebras and separable alternative algebras was established, and the conditions under which an Okubo algebra has nonzero idempotents were determined, again with $\chrs \F \neq 2, 3$. Similar results hold also for the case when $\chrs \F = 2$, see the paragraph after Theorem~1.1 in~\cite{Elduque1}. Finally, in~\cite[Theorem~5.1]{Elduque1} all symmetric composition algebras over a field~$\F$ with $\chrs \F = 3$ were explicitly classified (both with and without nonzero idempotents). The current version of Theorem~\ref{theorem:flexible-composition-classification} was obtained by Elduque and Myung in~\cite{ElduqueMyung3}. Moreover, they showed that the condition of flexibility can be replaced with a weaker condition of strict third power associativity, see~\cite[Theorem~4.5]{ElduqueMyung3}.

\begin{theorem}[{\cite[Theorem~3.2]{ElduqueMyung3}}] \label{theorem:flexible-composition-classification}
Any finite-dimensional flexible composition algebra~$\A$ over a field~$\F$ of arbitrary characteristic is either unital, i.e., Hurwitz algebra, or symmetric (see Corollary~\ref{corollary:symmetric-composition-classification}).
\end{theorem}

For more information on symmetric composition algebras, see~\cite{Elduque4,Elduque6} and \cite[Chapter~VIII]{Knus}.

\subsection{Lengths of Okubo algebras with nonzero idempotents or zero divisors} \label{subsection:Okubo-length}

\begin{proposition}[{\cite[Lemma~2.1]{ElduqueMyung1}}] \label{proposition:zero-divisors}
Let~$\A$ be a finite-dimensional composition algebra. A nonzero element $a \in \A$ is a zero divisor if and only if $n(a) = 0$.
\end{proposition}

In this subsection we compute the length of an arbitrary Okubo algebra $\Okubo$ with nonzero idempotents or zero divisors. Proposition~\ref{proposition:zero-divisors} implies that existence of zero divisors is equivalent to the fact that the norm on $\Okubo$ is isotropic. Hence, by Theorems~\ref{theorem:Okubo-isotropic-norm-classification} and~\ref{theorem:symmetric-idempotents-classification}, in this case the multiplication on $\Okubo$ is given either by Table~\ref{table:okubo-algebra-isotropic} or by Table~\ref{table:okubo-algebra-idempotents}. To shorten notation, we replace $*$-multiplication with concatenation, that is, we rewrite $x * y$ simply as~$xy$.

\begin{example} \label{example:Okubo-alternative}
$\Okubo$ is not descendingly alternative:
\begin{itemize}
    \item In Table~\ref{table:okubo-algebra-isotropic} we take $a = x_{1,0}$ and $b = x_{0,-1}$. It holds that $aa = -\alpha x_{-1,0}$, $ab = x_{1,-1}$, $ba = 0$, and $a(ab) = \alpha x_{-1,-1}$. Then $a(ab) \notin \Lin_1(a,b,aa,ab,ba)$.
    \item In Table~\ref{table:okubo-algebra-idempotents} we take $a = x_3$ and $b = x_6$. We have $aa = \beta x_0$, $ab = \beta x_5$, $ba = -\beta x_4$, and $a(ab) = \beta x_7$. Clearly, $a(ab) \notin \Lin_1(a,b,aa,ab,ba)$.
\end{itemize}
\end{example}

\begin{example} \label{example:Okubo-length}
There are $a,b \in \Okubo$ such that $S = \{ a,b \}$ is generating for $\Okubo$ and $l(S) = 4$. Since $\Okubo$ is descendingly flexible, we have $x(yx), (xy)x, (xy)z + (zy)x, x(yz) + z(yx) \in \Lin_2(S)$ for all $x,y,z \in S$. Hence
\begin{align*}
    \Lin_1(S) &{}= \Lin(\{a,b\});\\
    \Lin_2(S) &{}= \Lin_1(S) + \Lin(\{aa,ab,ba,bb\});\\
    \Lin_3(S) &{}= \Lin_2(S) + \Lin(\{a(ab),a(bb),(ba)a,(bb)a\});\\
    \Lin_4(S) &{}= \Lin_3(S) + \Lin(S^{(4)}).
\end{align*}
\begin{itemize}
    \item In Table~\ref{table:okubo-algebra-isotropic} we take $a = x_{0,1}$, $b = x_{1,0}$. Then
    \begin{align*}
        aa &{}= -\beta x_{0,-1}, & ba &{}= 0, & a(ab) &{}= \beta x_{1,-1}, \\
        bb &{}= -\alpha x_{-1,0}, & a(bb) &{}= 0, & (bb)a &{}= -\alpha x_{-1,1}, \\
        ab &{}= x_{1,1}, & (ba)a &{}= 0, & (aa)(bb) &{}= (\alpha \beta) x_{-1,-1}.
    \end{align*}
    Hence $\Lin_3(S) \neq \Okubo$ but $\Lin_4(S) = \Okubo$. Therefore, $S$ is generating for $\Okubo$, and $l(S) = 4$.
    \item In Table~\ref{table:okubo-algebra-idempotents} we take $a = x_1$, $b = x_3 + x_7$. Then 
    \begin{align*}
        aa &{}= x_1 = a, & a(ab) &{}= -x_2 - x_3 + x_7 = -ba,\\
        ab &{}= -x_2 - x_7, & a(bb) &{}= -\beta x_0 - (\beta + \beta \gamma) x_1 - 2 \beta x_4 - \beta x_5,\\
        ba &{}= x_2 + x_3 - x_7, & (ba)a &{}= x_2 + x_7 = -ab,\\
        bb &{}= \beta x_0 - \beta \gamma x_1 - \beta x_4 - 2\beta x_5, & (bb)a &{}= -\beta x_0 - (\beta + \beta \gamma) x_1 + \beta x_4 - \beta x_5.
    \end{align*}
    In this case $\chrs \F \neq 3$, so the elements $b, ab, ba$ are linearly independent, and thus $\Lin(b,ab,ba) = \Lin(x_2, x_3, x_7)$. We also have $(bb)a - a(bb) = 3\beta x_4$, so 
    $$
    \Lin(a, bb, a(bb), (bb)a) = \Lin(x_1, x_4, \beta x_0 - 2\beta x_5, -\beta x_0 - \beta x_5) = \Lin(x_0, x_1, x_4, x_5).
    $$
    Therefore, $\Lin_3(S) = \Lin(x_0, x_1, x_2, x_3, x_4, x_5, x_7)$. Consider now 
    $$
    ((bb)a)b = (\beta + 2 \beta \gamma) (x_2 + x_3) + 3\beta x_6 + (2 \beta + \beta \gamma) x_7.
    $$
    Since the coefficient at $x_6$ is nonzero, we have $\Lin_4(S) \supseteq \Lin_3(S) + \Lin(((bb)a)b) = \Okubo$. Hence $S$ is generating for $\Okubo$, and $l(S) = 4$.
\end{itemize}
\end{example}

\begin{theorem} \label{theorem:Okubo-length}
Let $\Okubo$ be an Okubo algebra over an arbitrary field~$\F$ with nonzero idempotents or zero divisors. Then $l(\Okubo) = 4$.
\end{theorem}

\begin{proof}
By Example~\ref{example:Okubo-length}, $l(\Okubo) \geq 4$. Since $\Okubo$ is descendingly flexible and $\dim \Okubo = 8$, Theorem~\ref{theorem:descendingly-flexible-length} implies that $l(\Okubo) \leq 4$. Hence $l(\Okubo) = 4$.
\end{proof}

\begin{remark} \label{remark:nice-field}
By Theorems~\ref{theorem:Okubo-algebras-isotropic} and~\ref{theorem:Okubo-algebras-finite}, if~$\F$ is finite, $\chrs \F = 3$ or~$\F$ contains the cubic roots of $1$, then the norm on any Okubo algebra $\Okubo$ over~$\F$ is isotropic. By Proposition~\ref{proposition:contains-idempotent}, if~$\F$ is either algebraically or real closed field, then any Okubo algebra $\Okubo$ over~$\F$ has a nonzero idempotent. In both cases we immediately have $l(\Okubo) = 4$.
\end{remark}

\subsection{Lengths of Okubo algebras without nonzero idempotents and zero divisors} \label{subsection:Okubo-length-without-idempotents}

In this subsection $\Okubo$ denotes an Okubo algebra without nonzero idempotents and zero divisors over an arbitrary field~$\F$. By Proposition~\ref{proposition:zero-divisors}, the norm on $\Okubo$ is anisotropic, so, by Theorem~\ref{theorem:Okubo-algebras-isotropic}, in this case $\chrs \F \neq 3$.

\begin{lemma} \label{lemma:Okubo-subalgebras}
\leavevmode
\begin{enumerate}[{\rm (1)}]
    \item Any nonzero $x \in \Okubo$ generates a two-dimensional subalgebra $\Lin(x, x^2)$.
    \item If $x, y \in \Okubo$ do not belong to the same two-dimensional subalgebra of $\Okubo$, i.e., $x \neq 0$ and $y \notin \Lin(x, x^2)$, then $x, x^2, y, y^2$ are linearly independent.
\end{enumerate}
\end{lemma}

\begin{proof}
\leavevmode
\begin{enumerate}[{\rm (1)}]
    \item Let $x \in \Okubo \setminus \{ 0 \}$. Then $x(xx) = (xx)x = n(x)x$, so $\Lin_3(\{ x \}) = \Lin_2(\{ x \})$, and thus, by Lemma~\ref{lemma:full-stabilization}, $\Lin_{\infty}(\{ x \}) = \Lin_2(\{ x \})$. In other words, the subalgebra generated by~$x$ coincides with $\Lin(x, x^2)$. Since there are no nonzero idempotents and zero divisors in~$\Okubo$, the elements $x$ and $x^2$ are linearly independent, so $\dim \Lin(x, x^2) = 2$.
    \item Assume to the contrary that $x, x^2, y, y^2$ are linearly dependent. Then there exists a nonzero element $z \in \Lin(x, x^2) \cap \Lin(y, y^2)$. It follows from~(1) that $\Lin(x, x^2) = \Lin(z, z^2) = \Lin(y, y^2)$, so $y \in \Lin(x, x^2)$, a contradiction.
    \qedhere
\end{enumerate}
\end{proof}

\begin{lemma} \label{lemma:Okubo-nondegenerate}
Let $\B \subseteq \Okubo$ be a nonzero subalgebra. Then either $\B = \Okubo$ or $\dim \B = 2$.
\end{lemma}

\begin{proof}
Since the norm on $\Okubo$ is anisotropic, the norm on~$\B$ is also anisotropic. Hence, if $\chrs \F \neq 2$, we immediately obtain that the norm on~$\B$ is strictly nondegenerate. Let now $\chrs \F = 2$. Assume from the contrary that the norm on~$\B$ is not strictly nondegenerate, that is, there exists some nonzero $x \in \B$ such that $n(x,y) = 0$ for all $y \in \B$. In particular, we can take $y = x^2$. Then, by Eq.~\eqref{equation:Okubo-flexibility-linearized}, we have $n(x)x^2 + (x^2)^2 = x(xx^2) + x^2(xx) = n(x,x^2)x = 0$, so $x^2$ and $(x^2)^2$ are linearly dependent. We obtain a contradiction with Lemma~\ref{lemma:Okubo-subalgebras}(1).

Therefore, in both cases the norm on~$\B$ is strictly nondegenerate, so~$\B$ is a composition algebra. Moreover, $\Okubo$ is a symmetric composition algebra without nonzero idempotents, so~$\B$ is also a symmetric composition algebra without nonzero idempotents. Since a para-unit is a nonzero idempotent in any para-Hurwitz algebra, it follows from Theorem~\ref{theorem:forms-of-para-Hurwitz} and Corollary~\ref{corollary:symmetric-composition-classification} that $\dim \B \in \{ 2, 8 \}$.
\end{proof}

\begin{corollary} \label{corollary:two-element-generating}
The set $S = \{ x, y \}$ is generating for $\Okubo$ if and only if $x, y$ do not belong to the same two-dimensional subalgebra of $\Okubo$.
\end{corollary}

\begin{proof}
The implication from left to right is obvious. Assume now that $x, y$ do not belong to the same two-dimensional subalgebra of $\Okubo$, i.e., $x \neq 0$ and $y \notin \Lin(x, x^2)$. Then the subalgebra generated by $S$ has dimension at least three, and thus, by Lemma~\ref{lemma:Okubo-nondegenerate}, coincides with $\Okubo$.
\end{proof}

\begin{lemma} \label{lemma:two-element-generating}
If $S$ is generating for $\Okubo$ and $|S| = 2$, then $l(S) = 3$.
\end{lemma}

\begin{proof}
Let $S = \{ a, b \}$. Since $S$ is generating for $\Okubo$, we have $a \neq 0$ and $b \notin \Lin(a, a^2)$, so Lemma~\ref{lemma:Okubo-subalgebras}(2) implies that $a, a^2, b, b^2$ are linearly independent. Similarly to Example~\ref{example:Okubo-length}, we obtain that
\begin{align*}
    \Lin_1(S) &{}= \Lin(\{a,b\});\\
    \Lin_2(S) &{}= \Lin_1(S) + \Lin(\{aa,ab,ba,bb\});\\
    \Lin_3(S) &{}= \Lin_2(S) + \Lin(\{a(ab),b(ba),(ba)a,(ab)b\});\\
    \Lin_4(S) &{}= \Lin_3(S) + \Lin(S^{(4)}).
\end{align*}
Assume first that $a, b, aa, ab, ba, bb$ are linearly independent. Then, in terms of Definition~\ref{definition:sequence-differences}, we have $d_1 = 2$ and $d_2 = 4$. If $d_3 \leq 1$, then, by Proposition~\ref{proposition:differences-descendingly-flexible}, we have $d_4 = 0$, so $d_k = 0$ for all $k \geq 4$. Hence $\sum \limits_{k = 0}^{\infty} d_k \leq 7 < \dim \Okubo$, and Proposition~\ref{proposition:differences-length}(1) then implies that $S$ is not generating for $\Okubo$, a contradiction. Therefore, in this case $d_3 = 2$, so $d_1 + d_2 + d_3 = 8 = \dim \Okubo$, and thus $l(S) = 3$.

Let now $a, b, aa, ab, ba, bb$ be linearly dependent, that is, 
\begin{equation} \label{equation:two-element-generating-1}
k_1 a + k_2 b + k_3 aa + k_4 bb + k_5 ab + k_6 ba = 0    
\end{equation}
for some $k_j \in \F$, $j = 1, \dots, 6$, at least one of which is nonzero. We may assume without loss of generality that $k_5 = 1$. Then $ab \in \Lin(a, b, aa, bb, ba)$, so 
\begin{align*}
a(ab) &{}\in \Lin(aa, ab, a(aa), a(bb), a(ba)) \subseteq \Lin_2(S) + \Lin(\{b(ba)\}),\\
(ab)b &{}\in \Lin(ab, bb, (aa)b, (bb)b, (ba)b) \subseteq \Lin_2(S) + \Lin(\{(ba)a\}).
\end{align*}
Hence 
\begin{equation} \label{equation:two-element-generating-4}
\Lin_3(S) = \Lin_2(S) + \Lin(\{b(ba),(ba)a\}).  
\end{equation}
Note that, by Eqs.~\eqref{equation:Okubo-flexibility-linearized} and~\eqref{equation:two-element-generating-1}, we have
\begin{equation} \label{equation:two-element-generating-2}
\begin{aligned}
(bb)a &{}= n(a,b)b - (ab)b \\
&{}= n(a,b)b + (k_1 a + k_2 b + k_3 aa + k_4 bb + k_6 ba)b \\
&{}= n(a,b)b + k_1 ab + k_2 bb + k_3 (aa)b + k_4 (bb)b + k_6 (ba)b \\
&{}= k_6 n(b)a + (n(a,b) + k_4 n(b)) b + k_1 ab + k_2 bb + k_3 (n(a,b)a - (ba)a) \\
&{}= (k_3 n(a,b) + k_6 n(b))a + (n(a,b) + k_4 n(b)) b + k_1 ab + k_2 bb - k_3 (ba)a.
\end{aligned}
\end{equation}
Multiplying Eq.~\eqref{equation:two-element-generating-1} by $a$ on the right and substituting the expression for $(bb)a$ from Eq.~\eqref{equation:two-element-generating-2}, we obtain that
\begin{equation} \label{equation:two-element-generating-3}
\begin{aligned}
0 &{}= k_1 aa + k_2 ba + k_3 (aa)a + k_4 (bb)a + (ab)a + k_6 (ba)a \\
&{}= k_1 aa + k_2 ba + k_3 n(a)a + n(a)b + k_6 (ba)a \\
&{}+ k_4 ((k_3 n(a,b) + k_6 n(b))a + (n(a,b) + k_4 n(b)) b + k_1 ab + k_2 bb - k_3 (ba)a) \\
&{}= (k_3n(a) + k_3 k_4 n(a,b) + k_4 k_6 n(b))a + (n(a) + k_4 n(a,b) + k_4^2 n(b)) b \\
&{}+ k_1 aa + k_2 ba + k_1 k_4 ab + k_2 k_4 bb + (k_6 - k_3 k_4) (ba)a.
\end{aligned}
\end{equation}
If $k_6 \neq k_3 k_4$, then $(ba)a \in \Lin_2(S)$. Similarly, multiplying Eq.~\eqref{equation:two-element-generating-1} by $b$ on the left, one can show that in this case $b(ba) \in \Lin_2(S)$. Then, by Eq.~\eqref{equation:two-element-generating-4}, we have $\Lin_3(S) = \Lin_2(S)$, so $d_1 = 2$, $d_2 \leq 3$, and $d_k = 0$ for all $k \geq 3$. Therefore, $S$ is not generating for $\Okubo$, a contradiction. Thus we have $k_6 = k_3 k_4$, and Eq.~\eqref{equation:two-element-generating-3} takes the form
\begin{equation} \label{equation:two-element-generating-5}
n(a + k_4 b)(k_3 a + b) + k_1 aa + k_2 ba + k_1 k_4 ab + k_2 k_4 bb = 0,
\end{equation}
since $n(a) + k_4 n(a,b) + k_4^2 n(b) = n(a) + n(a, k_4 b) + n(k_4 b) = n(a + k_4 b)$.

Assume that $k_4 = 0$. Then it follows from Eqs.~\eqref{equation:two-element-generating-1} and~\eqref{equation:two-element-generating-5} that
\begin{align}
k_1 a + k_2 b + k_3 aa + ab &{}= 0, \label{equation:two-element-generating-6}\\
n(a)(k_3 a + b) + k_1 aa + k_2 ba &{}= 0. \label{equation:two-element-generating-7}
\end{align}
If $k_2 = 0$, then Eq.~\eqref{equation:two-element-generating-7} implies that $b \in \Lin(a, a^2)$, a contradiction. Hence $k_2 \neq 0$, and thus it follows from Eqs.~\eqref{equation:two-element-generating-6} and~\eqref{equation:two-element-generating-7} that $ab, ba \in \Lin(a, b, aa)$. Then $(ba)a \in \Lin(aa, ba, (aa)a) \subseteq \Lin_2(S)$ and $b(ba) \in \Lin(b(aa)) + \Lin_2(S) = \Lin(a(ab)) + \Lin_2(S) = \Lin(a(aa)) + \Lin_2(S) = \Lin_2(S)$. By Eq.~\eqref{equation:two-element-generating-4}, we have $\Lin_3(S) = \Lin_2(S)$, so $S$ is not generating for $\Okubo$, a contradiction. Therefore, $k_4 \neq 0$ and, similarly, $k_3 \neq 0$.

As a consequence, Eq.~\eqref{equation:two-element-generating-5} must be proportional to Eq.~\eqref{equation:two-element-generating-1}, since otherwise we can multiply Eq.~\eqref{equation:two-element-generating-5} by some $\alpha \in \F$ and subtract it from Eq.~\eqref{equation:two-element-generating-1}, so that either exactly one of the coefficients at $ab$ and $ba$ is nonzero, or both coefficients at $ab$ and $ba$ are zero, but some nontrivial linear combination of $a, a^2, b, b^2$ is zero.

Let us denote $p = n(a + k_4 b)$. Then we have
\begin{equation} \label{equation:two-element-generating-8}
\rk
\begin{pmatrix}
k_1 & k_2 & k_3 & k_4 & 1 & k_3 k_4\\
k_3 p & p & k_1 & k_2 k_4 & k_1 k_4 & k_2
\end{pmatrix}
= 1.
\end{equation}
Since $k_4 \neq 0$, we immediately obtain from the fourth and the fifth columns that $k_2 = k_1 k_4$, so $k_1$ and $k_2$ are equal to zero or not equal to zero simultaneously. Assume that $k_1 = k_2 = 0$. Then Eq.~\eqref{equation:two-element-generating-1} implies that
$$
(a + k_4 b)(k_3 a + b) = k_3 aa + k_4 bb + ab + k_3 k_4 ba = 0.
$$
But $a + k_4 b$ and $k_3 a + b$ are nonzero, and there are no zero divisors in $\Okubo$, a contradiction. Hence $k_1 \neq 0$ and $k_2 \neq 0$. It then follows from the fifth and the sixth columns in Eq.~\eqref{equation:two-element-generating-8} that $k_3 k_4 = 1$. Multiplying  Eq.~\eqref{equation:two-element-generating-1} by $k_4$, we obtain that
\begin{align*}
    0 &{}= k_1 k_4 a + k_1 k_4^2 b + aa + k_4^2 bb + k_4 ab + k_4 ba \\
    &{}= k_1 k_4 (a + k_4 b) + (a + k_4 b)^2.
\end{align*}
But $a + k_4 b \neq 0$, so $a + k_4 b$ and $(a + k_4 b)^2$ are linearly independent, a contradiction. Therefore, $a, b, aa, ab, ba, bb$ must be linearly independent.
\end{proof}

\begin{theorem} \label{theorem:Okubo-length-exceptional}
Let $\Okubo$ be an Okubo algebra over an arbitrary field~$\F$ without nonzero idempotents and zero divisors. Then $l(\Okubo) = 3$.
\end{theorem}

\begin{proof}
By Corollary~\ref{corollary:two-element-generating} and Lemma~\ref{lemma:two-element-generating}, there exists a generating system $S = \{ x, y \}$ for~$\Okubo$ with $l(S) = 3$, so $l(\Okubo) \geq 3$.

Let now $S$ be an arbitrary generating system for $\Okubo$. Then there exist $x, y \in S$ which do not belong to the same two-dimensional subalgebra of $\Okubo$, since otherwise $S$ is contained in a two-dimensional subalgebra generated by any nonzero element from $S$. By Corollary~\ref{corollary:two-element-generating} and Lemma~\ref{lemma:two-element-generating}, the system $S' = \{ x, y \} \subseteq S$ is generating for $\Okubo$, and $l(S) \leq l(S') = 3$. Thus $l(\Okubo) = 3$.
\end{proof}

\begin{proposition} \label{proposition:non-para-Hurwitz}
Let~$\A$ be a two-dimensional algebra from Theorem~\ref{theorem:forms-of-para-Hurwitz}(2). Then ${l(\A) = 2}$.
\end{proposition}

\begin{proof}
Clearly, $l(\A) \leq \dim \A = 2$ and $S = \{ u \}$ is generating for~$\A$ with ${l(S) = 2}$.
\end{proof}

\begin{corollary} \label{corollary:all-lengths}
Together with Theorem~\ref{theorem:standard-length}, where the lengths of Hurwitz and para-Hurwitz algebras were computed, Theorems~\ref{theorem:Okubo-length} and~\ref{theorem:Okubo-length-exceptional} and Proposition~\ref{proposition:non-para-Hurwitz} give a complete description of lengths of symmetric composition algebras and finite-dimensional flexible composition algebras over an arbitrary field~$\F$.
\end{corollary}

\begin{proof}
By Theorem~\ref{theorem:flexible-composition-classification}, if~$\A$ is a finite-dimensional flexible composition algebra, then~$\A$ is either a Hurwitz algebra or a symmetric composition algebra. In the latter case, by Corollary~\ref{corollary:symmetric-composition-classification},~$\A$ is either an Okubo algebra or a form of a para-Hurwitz algebra. By Theorem~\ref{theorem:forms-of-para-Hurwitz}, the second case implies that~$\A$ is either a para-Hurwitz algebra or a two-dimensional algebra from Theorem~\ref{theorem:forms-of-para-Hurwitz}(2). The length of~$\A$ is known in all cases.
\end{proof}

\end{document}